\documentclass[amssymb, 11pt]{amsart}
\usepackage{latexsym}
\usepackage[margin=.5in]{geometry}

\newlength{\standardunitlength}
\newtheorem{prop}{Proposition}[section]
\newtheorem{definition}[prop]{Definition}

\newtheorem{lemma}[prop]{Lemma}

\newtheorem{theorem}[prop]{Theorem}

\begin{document}

\title [Derangements in affine] {Derangements in affine classical groups and Cohen-Lenstra heuristics}

\author{Jason Fulman}
\address{Department of Mathematics, University of Southern California, Los Angeles, CA 90089-2532, USA}
\email{fulman@usc.edu}

\author{Dennis Stanton}
\address{Department of Mathematics, University of Minnesota, Minneapolis, MN 55455, USA}
\email{stant001@umn.edu}



\date{October 16, 2025}

\begin{abstract} We observe that Anzanello's work on the proportion of derangements in affine classical groups over finite fields
is related to symplectic and orthogonal Cohen-Lenstra type distributions on integer partitions. This leads
to a proof of three q-polynomial identities conjectured by Anzanello, which were crucial for her work.
\end{abstract}

\maketitle

\section{Introduction} Let $G$ be a finite permutation group acting transitively on a set $X$ of size greater than 1.
An element of $G$ is called
a derangement if it has no fixed points on $X$. It is of great interest to study the proportion of elements of $G$ which are
derangements. Indeed Serre \cite{Se} and Section 8 of \cite{DFG} survey connections of derangements with topology, number
theory, and other parts of mathematics. It is elementary (a theorem of Jordan) that the proportion of derangements of $G$
is positive. It is also known (conjectured by Boston and Shalev and proved in a series of four papers of Fulman and Guralnick)
that if $G$ is simple then the proportion of derangements is bounded away from $0$ by a universal constant. It is also  
of interest to enumerate derangements satisfying specific properties, such as derangements of prescribed order. In fact
an entire book \cite{BG} is devoted to such problems; see Chapter 1 of \cite{BG} for an overview.

Building on work of Pablo Spiga for affine general linear groups, Anzanello \cite{An} studies the proportion of derangements
and derangements of p-power order for finite affine unitary groups and odd characteristic finite affine symplectic and orthogonal
groups. Her results for the unitary groups are complete, but her results for odd characteristic affine symplectic and orthogonal
groups depend on the following three q-polynomial identities which she conjectured (the notation is defined in Section \ref{back}):

\begin{equation} \label{anz1} \sum_{|\lambda|=2m \atop i \ odd \implies m_i(\lambda) \ even} \frac{1-q^{-\lambda_1'}}
{q^{\frac{1}{2} \sum_i (\lambda_i')^2 + \frac{o(\lambda)}{2}} \prod_i (1/q^2;1/q^2)_{\lfloor \frac{m_i(\lambda)}{2} \rfloor}}
 =  \frac{1}{q^m (q+1)} \sum_{i=1}^m \frac{(-1)^{i-1} (q^{2i+1}+1)}{q^{i(i+1)} (1/q^2;1/q^2)_{m-i}}. \end{equation}

\begin{eqnarray} \label{anz2} & &  \sum_{|\lambda|=2m+1 \atop i \ even \implies m_i(\lambda) \ even} \frac{1-q^{-\lambda_1'}}
{q^{\frac{1}{2} \sum_i (\lambda_i')^2 - \frac{o(\lambda)}{2}} \prod_i (1/q^2;1/q^2)_{\lfloor \frac{m_i(\lambda)}{2} \rfloor}}
\end{eqnarray}
 \[  =  \frac{1}{q^m (1/q^2;1/q^2)_m} + \frac{1}{q^{m+1}} \sum_{i=0}^m \frac{(-1)^{i-1}}{q^{i(i+1)} (1/q^2;1/q^2)_{m-i}}.\]

\begin{equation} \label{anz3} \sum_{|\lambda|=2m \atop i \ even \implies m_i(\lambda) \ even} \frac{1-q^{-\lambda_1'}}
{q^{\frac{1}{2} \sum_i (\lambda_i')^2 - \frac{o(\lambda)}{2}} \prod_i (1/q^2;1/q^2)_{\lfloor \frac{m_i(\lambda)}{2} \rfloor}} 
 =  \frac{1}{q^m} \sum_{i=1}^m \frac{(-1)^{i-1}}{q^{i(i-1)} (1/q^2;1/q^2)_{m-i}}. \end{equation}

The purpose of this paper is to prove Conjectures \eqref{anz1}, \eqref{anz2}, \eqref{anz3}, by connecting them to work on
symplectic and orthogonal Cohen-Lenstra type distributions introduced in \cite{F1} (see also \cite{F2} which describes these
using notation similar to that of Anzanello). We find this connection to be interesting (possibly more interesting than the conjectures
themselves) and it would be nice to have a
conceptual reason for it. Cohen-Lenstra heuristics were introduced in \cite{CL} in a number theoretic context and have been
studied by giants (see the surveys \cite{EV},\cite{Wood}). In fact the paper \cite{CL} has 843 google scholar citations at
the current time. Our proof of Conjectures \eqref{anz1}, \eqref{anz2}, \eqref{anz3} also uses some results for
terminating $_2 \phi _1$ basic hypergeometric series.

This paper is organized as follows. The (short) Section \ref{back} collects some notation and gives needed background
on hypergeometric series. Section \ref{Spproof} proves Anzanello's first conjecture by linking it with a symplectic
Cohen-Lenstra distribution on the set of integer partitions which have all odd parts occurring with even multiplicity. Sections
\ref{Oproof1} and \ref{Oproof2} prove Anzanello's second and third conjectures by linking them with an orthogonal
Cohen-Lenstra distribution on the set of integer partitions which have all even parts occurring with even multiplicity.

\section{Notation and terminating $_2 \phi _1$ facts} \label{back}
 
Throughout this paper we use the notation that
\[ (A;Q)_n = (1-A)(1-AQ) \cdots (1-AQ^{n-1}) = \prod_{k=0}^{n-1} (1-AQ^k).\] For $\lambda$ an integer partition,
we let $|\lambda|$ denote the size of $\lambda$, and let $m_i(\lambda)$ denote the number of parts of $\lambda$ of size $i$.
We let $o(\lambda)$ denote the number of odd parts of $\lambda$. We let $\lambda_i' = \sum_{j \geq i} m_j(\lambda)$ (so $\lambda_1'$
is the number of parts of $\lambda$ and $\lambda_i'$ is the size of the ith column of the diagram of $\lambda$).

Next we discuss some facts about terminating  $_2 \phi _1$ series, beginning with a definition.

\begin{definition} For a non-negative integer $n$ let
$$
\ _2\phi_1(Q^{-n},B;C;Q,Z)=
\sum_{k=0}^n \frac{(Q^{-n};Q)_k (B;Q)_k}{(Q;Q)_k (C;q)_k} Z^k
$$ \end{definition}

We need an evaluation of a terminating $\ _2\phi_1$, see \cite[II.7]{GR}.

\begin{prop} 
\label{qchu}
For a non-negative integer $n$, we have
$$
\ _2\phi_1(Q^{-n},B;C;Q,CQ^n/B)= \frac{(C/B;Q)_n}{(C;Q)_n}.
$$
\end{prop}

We need a transformation of a terminating $\ _2\phi_1$, see \cite[III.7]{GR}.
\begin{prop} 
\label{term2phi1trans}
For a non-negative integer $n$, we have
$$
\ _2\phi_1(Q^{-n},B;C;Q,Z)= \frac{(C/B;Q)_n}{(C;Q)_n} 
\sum_{k=0}^n \frac{(Q^{-n};Q)_k (B;Q)_k (BZQ^{-n}/C;Q)_k}{(Q;Q)_k (BQ^{1-n}/C;Q)_k} Q^k.
$$
\end{prop}

This limiting case of Proposition~\ref{term2phi1trans} will be used.

\begin{prop} 
\label{limitterm2phi1trans}
For a non-negative integer $n$, we have
$$
\begin{aligned}
\lim_{B\to\infty} \ _2\phi_1(Q^{-n},B;C;Q,Z/B) &= 
\sum_{k=0} ^n \frac{(Q^{-n};Q)_k (-1)^k Q^{\binom{k}{2}}}{(Q;Q)_k (C;q)_k} Z^k\\
&=\frac{1}{(C;Q)_n} 
\sum_{k=0}^n \frac{(Q^{-n};Q)_k  (ZQ^{-n}/C;Q)_k}{(Q;Q)_k} Q^k (CQ^{n-1})^k.\\
\end{aligned}
$$
\end{prop}

Finally we need the $Q$-binomial theorem, see \cite[II.3]{GR}.
\begin{theorem} 
\label{qbinth}
We have
$$
\frac{(AZ;Q)_\infty}{(Z;Q)_\infty} =\sum_{k=0}^\infty \frac{(A;Q)_k}{(Q;Q)_k} Z^k.
$$
\end{theorem}

\section{Anzanello's first conjecture} \label{Spproof}

The paper \cite{F1} (see also Section 4 of \cite{F2} which uses notation similar to that of Anzanello), defined and studied the following
probability distribution on integer partitions for which all odd parts occur with even multiplicity:

\[ P_{Sp}(\lambda) = \prod_{i \geq 1} (1-u^2/q^{2i-1}) \frac{u^{|\lambda|}}{q^{\frac{1}{2} \sum_i (\lambda_i')^2 + \frac{o(\lambda)}{2}}  
\prod_i (1/q^2;1/q^2)_{\lfloor \frac{m_i(\lambda)}{2} \rfloor}}.\] 

In order to give an algorithm for generating partitions distributed according to $P_{Sp}$, it was proved by a recursive argument in \cite{F1} that

\[ \sum_{\lambda_1'=2k} P_{Sp}(\lambda) =  \prod_{i \geq 1} (1-u^2/q^{2i-1}) \frac{u^{2k}}{q^{2k^2+k} (1/q^2;1/q^2)_k (u^2/q;1/q^2)_k} \] and that
\[ \sum_{\lambda_1'=2k-1} P_{Sp}(\lambda) = \prod_{i \geq 1} (1-u^2/q^{2i-1}) \frac{u^{2k}}{q^{2k^2-k} (1/q^2;1/q^2)_{k-1}
(u^2/q;1/q^2)_{k}}.\]

It follows that \begin{eqnarray*}
& & \sum_{|\lambda|=2m, \lambda_1'=2k \atop i \ odd \implies m_i(\lambda) \ even} \frac{1}
{q^{\frac{1}{2} \sum_i (\lambda_i')^2 + \frac{o(\lambda)}{2}} \prod_i (1/q^2;1/q^2)_{\lfloor \frac{m_i(\lambda)}{2} \rfloor}}\\
& = & \frac{1}{q^{2k^2+k} (1/q^2;1/q^2)_k} \times \rm{Coeff. \ } u^{m-k} \rm{\ in \ } \frac{1}{(u/q;1/q^2)_k}
\end{eqnarray*} and that
 \begin{eqnarray*} & & \sum_{|\lambda|=2m, \lambda_1'=2k-1 \atop i \ odd \implies m_i(\lambda) \ even} \frac{1}
{q^{\frac{1}{2} \sum_i (\lambda_i')^2 + \frac{o(\lambda)}{2}} \prod_i (1/q^2;1/q^2)_{\lfloor \frac{m_i(\lambda)}{2} \rfloor}} \\
& = & \frac{1}{q^{2k^2-k} (1/q^2;1/q^2)_{k-1}} \times \rm{Coeff. \ } u^{m-k} \rm{\ in \ } \frac{1}{(u/q;1/q^2)_k}. \end{eqnarray*}

Thus the left side of Conjecture \eqref{anz1} is equal to $\sum_{k=1}^m (a_k+b_k)$, where
\[ a_k = \frac{1}{q^{2k^2+k} (1/q^2;1/q^2)_{k-1}} \times  \rm{Coeff. \ } u^{m-k} \rm{\ in \ } \frac{1}{(u/q;1/q^2)_k} \]  and 
\[ b_k = \left( 1 - \frac{1}{q^{2k-1}} \right) \frac{1}{q^{2k^2-k} (1/q^2;1/q^2)_{k-1}} \times
 \rm{Coeff. \ } u^{m-k} \rm{\ in \ } \frac{1}{(u/q;1/q^2)_k}.\] Thus to prove the Conjecture \eqref{anz1}, it is necessary to show
that \begin{equation} \label{maybe4} \sum_{k=1}^m (a_k + b_k) = \frac{1}{q^m (q+1)} \sum_{i=1}^m \frac{(-1)^{i-1} (q^{2i+1}+1)}{q^{i(i+1)} (1/q^2;1/q^2)_{m-i}}. \end{equation}

To prove \eqref{maybe4}, we will split the term $a_k+b_k = a2_k + b2_k$ into two other terms, applying the transformation in
Proposition \ref{limitterm2phi1trans} to each new term, and recombine.

\begin{lemma} For $m\ge k$ the coefficient of $u^{m-k}$ in 
$$
\frac{1}{(u/q;q^{-2})_k}
$$
is
$$
\frac{(q^{-2k};q^{-2})_{m-k}}{(q^{-2};q^{-2})_{m-k}} q^{-m+k}.
$$
\end{lemma}
\begin{proof} This follows from the Theorem~\ref{qbinth}, 
$Q=q^{-2}$, $A=Q^k$, $Z=u/q$, $s=m-k,$
$$
\frac{1}{(Z;Q)_k}=\frac{(ZQ^k;Q)_\infty}{(Z;Q)_\infty}=
\sum_{s=0}^\infty \frac{(Q^k;Q)_s}{(Q;Q)_s} Z^s.
$$
\end{proof}

\begin{definition} Fix $m\ge 1.$ For $1\le k\le m$ let
$$
a2_k= (1-q)\frac{q^{-2k^2-k}}{ (q^{-2};q^{-2})_{k-1}}\frac{(q^{-2k};q^{-2})_{m-k}}{(q^{-2};q^{-2})_{m-k}} q^{-m+k}.
$$
\end{definition}

\begin{definition} Fix $m\ge 1.$ For $1\le k\le m$ let
$$
b2_k= \frac{q^{-2k^2+k}}{ (q^{-2};q^{-2})_{k-1}}\frac{(q^{-2k};q^{-2})_{m-k}}{(q^{-2};q^{-2})_{m-k}} q^{-m+k}.
$$
\end{definition}

The next two Propositions give the basic hypergeometric forms
 for the $a2$ and $b2$ sums.

\begin{prop}
\label{a2}
We have
$$
\begin{aligned}
\sum_{k=1}^m a2_k&= q^{-m}(1-q)\sum_{s=0}^{m-1} (-1)^s \frac{(q^{2m-2};q^{-2})_s}{(q^{-2};q^{-2})_s^2}
q^{-s^2-3s-2sm-2}\\
&= q^{-m-2}(1-q)\lim_{b\to\infty} \ _2\phi_1(q^{2m-2},b;q^{-2};q^{-2},q^{-2m-4}/b).
\end{aligned}
$$

\end{prop} 

\begin{prop} We have
\label{b2}
$$
\begin{aligned}
\sum_{k=1}^m b2_k&= q^{-m}\sum_{s=0}^{m-1} (-1)^s \frac{(q^{2m-2};q^{-2})_s}{(q^{-2};q^{-2})_s^2}
q^{-s^2-s-2sm}\\
&= q^{-m}\lim_{b\to\infty} \ _2\phi_1(q^{2m-2},b;q^{-2};q^{-2},q^{-2m-2}/b).
\end{aligned}
$$
\end{prop}

We use Proposition~\ref{limitterm2phi1trans} to give alternative sums for $a2$ and $b2$.

\begin{prop} 
\label{alta2}
We have
$$
\sum_{k=1}^m a2_k =  \frac{1}{q^{m}(1+q)}\sum_{i=1}^m 
\frac{(-1)^{i-1}q^{-i(i+1)} (1-q^{2i})}{(1/q^2;1/q^2)_{m-i}}.
$$
\end{prop}

\begin{proof} 
Choose 
$$Q=q^{-2}, \quad n=m-1, \quad C=q^{-2}\quad Z=q^{-2m-4} $$
in Proposition~\ref{limitterm2phi1trans}, and apply it to Proposition~\ref{a2}.
$$
\begin{aligned}
\sum_{k=1}^m a2_k &= \frac{q^{-m-2}(1-q) }{(q^{-2};q^{-2})_{m-1}}
 \ _2\phi_1(q^{2m-2},q^{-4};0;q^{-2}, q^{-2m})\\
 &= 
 \frac{q^{-m-2}(1-q) }{(q^{-2};q^{-2})_{m-1}}
\sum_{s=0}^{m-1} \frac{(q^{2m-2};q^{-2})_s (q^{-4};q^{-2})_s}
{(q^{-2};q^{-2})_s} q^{-2ms}\\
&= 
 \frac{q^{-m-2}(1-q) }{(q^{-2};q^{-2})_{m-1}}
\sum_{s=0}^{m-1} \frac{(q^{2m-2};q^{-2})_s (1-q^{-2s-2})}
{(1-q^{-2})} q^{-2ms}\\
&= 
q^{-m-2}(1-q) 
\sum_{s=0}^{m-1} \frac{(q^{2-2m};q^{2})_s (-1)^s q^{(2m-2)s-s(s-1)}}{(q^{-2};q^{-2})_{m-1}} 
\frac{(1-q^{2s+2})}
{(1-q^{2})} q^{-2ms-2s}\\
&= 
\frac{q^{-m-2}}{1+q} 
\sum_{s=0}^{m-1} \frac{(-1)^s}{(q^{-2};q^{-2})_{m-1-s}} 
(1-q^{2s+2}) q^{-s^2-3s}\\
&= 
\frac{q^{-m}}{1+q} 
\sum_{i=1}^{m} \frac{(-1)^{i-1}}{(q^{-2};q^{-2})_{m-i}} 
(1-q^{2i}) q^{-i^2-i}.\\
\end{aligned}
$$

\end{proof} 

\begin{prop} 
\label{altb2}
We have
$$
\sum_{k=1}^m b2_k = q^{-m}\sum_{i=1}^m 
\frac{(-1)^{i-1}q^{-i(i+1)}}{(1/q^2;1/q^2)_{m-i}} q^{2i}.
$$
\end{prop}

\begin{proof} 
Choose 
$$
Q=q^{-2}, \quad n=m-1, \quad C=q^{-2}\quad Z=q^{-2m-2} 
$$
in Proposition~\ref{limitterm2phi1trans} and apply it to Proposition~\ref{b2}.

$$
\begin{aligned}
\sum_{k=1}^m b2_k &= \frac{q^{-m} }{(q^{-2};q^{-2})_{m-1}}
 \ _2\phi_1(q^{2m-2},q^{-2};0;q^{-2}, q^{-2m})\\
 &= 
 \frac{q^{-m}}{(q^{-2};q^{-2})_{m-1}}
\sum_{s=0}^{m-1} (q^{2m-2};q^{-2})_s 
 q^{-2ms}\\
&= 
q^{-m} 
\sum_{s=0}^{m-1} \frac{(q^{2-2m};q^{2})_s (-1)^s q^{(2m-2)s-s(s-1)}}{(q^{-2};q^{-2})_{m-1}} 
q^{-2ms}\\
&= 
q^{-m}
\sum_{s=0}^{m-1} \frac{(-1)^s}{(q^{-2};q^{-2})_{m-1-s}} 
q^{-s^2-s}\\
&= 
q^{-m}
\sum_{i=1}^{m} \frac{(-1)^{i-1}}{(q^{-2};q^{-2})_{m-i}} 
 q^{-i^2+i}.\\
\end{aligned}
$$
\end{proof} 

Finally, we can complete the proof of Conjecture \eqref{anz1}: from
Propositions~\ref{alta2} and \ref{altb2} we have
$$
\begin{aligned}
 \sum_{k=1}^m (a_k+b_k)= \sum_{k=1}^m (a2_k+b2_k)&= q^{-m}
\sum_{i=1}^m 
\frac{(-1)^{i-1}q^{-i(i+1)}}{(1/q^2;1/q^2)_{m-i}} 
\biggl( \frac{1-q^{2i}}{1+q}+q^{2i}\biggr)\\
&= \frac{q^{-m}}{1+q}
\sum_{i=1}^m 
\frac{(-1)^{i-1}q^{-i(i+1)}}{(1/q^2;1/q^2)_{m-i}} 
(q^{2i+1}+1).
\end{aligned}
$$.

\section{Anzanello's second conjecture} \label{Oproof1}

The paper \cite{F1} (see also Section 5 of \cite{F2} which uses notation similar to that of Anzanello) defined and studied
the following probability distribution on integer partitions for which the even parts occur with even multiplicity:

\[ P_O(\lambda) =  \frac{\prod_{i \geq 1} (1-u^2/q^{2i-1})}{(1+u)}  \frac{u^{|\lambda|}}{q^{\frac{1}{2} \sum_i (\lambda_i')^2 - \frac{o(\lambda)}{2}}  
\prod_i (1/q^2;1/q^2)_{\lfloor \frac{m_i(\lambda)}{2} \rfloor}}.\] 

In order to give an algorithm for generating partitions distributed according to $P_O$, it was proved in \cite{F1} that
\[ \sum_{\lambda_1'=2k-1} P_O(\lambda) = \frac{\prod_{i \geq 1} (1-u^2/q^{2i-1})}{(1+u)}  \frac{u^{2k-1}}{q^{2k^2-3k+1} (1/q^2;1/q^2)_{k-1} 
(u^2/q;1/q^2)_k}.\]

It follows that

\begin{eqnarray*}
& & \sum_{|\lambda|=2m+1, \lambda_1'=2k-1 \atop i \ even \implies m_i(\lambda) \ even} \frac{1}{q^{\frac{1}{2} \sum_i (\lambda_i')^2 - \frac{o(\lambda)}{2}}  
\prod_i (1/q^2;1/q^2)_{\lfloor \frac{m_i(\lambda)}{2} \rfloor}} \\
& = & \frac{1}{q^{2k^2-3k+1} (1/q^2;1/q^2)_{k-1}} \times \rm{Coeff. \ } u^{2m+1} \rm{\ in \ }  \frac{u^{2k-1}}{(u^2/q;1/q^2)_k} \\
& = & \frac{1}{q^{2k^2-3k+1} (1/q^2;1/q^2)_{k-1}} \times \rm{Coeff. \ } u^{m-k+1} \rm{\ in \ }  \frac{1}{(u/q;1/q^2)_k}.
\end{eqnarray*}

Note that if $|\lambda|$ is odd and $P_O(\lambda) \neq 0$, then since the even parts of $\lambda$ occur with even multiplicity, $\lambda_1'$
must be odd. Thus the left side of Conjecture \eqref{anz2} is equal to $\sum_{k=1}^{m+1} c_k$, where
\[ c_k = \left( 1 - \frac{1}{q^{2k-1}} \right) \frac{1}{q^{2k^2-3k+1} (1/q^2;1/q^2)_{k-1}} \times \rm{Coeff. \ } u^{m-k+1} \rm{\ in \ } 
\frac{1}{(u/q;1/q^2)_k}.\]

Thus to prove Conjecture \eqref{anz2} it is enough to show that
\begin{equation} \label{eq5} 
\sum_{k=1}^{m+1} c_k =  \frac{1}{q^m (1/q^2;1/q^2)_m} + \frac{1}{q^{m+1}} \sum_{i=0}^m \frac{(-1)^{i-1}}{q^{i(i+1)} (1/q^2;1/q^2)_{m-i}}.
\end{equation}

To prove \eqref{eq5}, we first write $c_k$ as the sum of two terms $c1_k + c2_k$. 

\begin{definition} Fix $m\ge 1.$ For $1\le k\le m$ let
$$
\begin{aligned}
c_k&= \bigl(1-1/q^{2k-1}\bigr)\frac{q^{-2k^2+3k-1}}{ (q^{-2};q^{-2})_{k-1}}\frac{(q^{-2k};q^{-2})_{m-k+1}}{(q^{-2};q^{-2})_{m-k+1}} q^{-1-m+k}\\
c_k&=c1_k+c2_k,
\end{aligned}
$$
where
$$
\begin{aligned}
c1_k&= \frac{q^{-2k^2+3k-1}}{ (q^{-2};q^{-2})_{k-1}}\frac{(q^{-2k};q^{-2})_{m-k+1}}{(q^{-2};q^{-2})_{m-k+1}} q^{-1-m+k},\\
c2_k&=-q^{1-2k} \frac{q^{-2k^2+3k-1}}{ (q^{-2};q^{-2})_{k-1}}\frac{(q^{-2k};q^{-2})_{m-k+1}}{(q^{-2};q^{-2})_{m-k+1}} q^{-1-m+k}.
\end{aligned}
$$
\end{definition} 

\begin{prop} 
\label{part01}
We have 
$$
\sum_{k=1}^{m+1} c2_k 
= q^{-m-1}\sum_{i=0}^{m} 
\frac{(-1)^{i-1}q^{-i(i+1)}}{(1/q^2;1/q^2)_{m-i}}
$$
which is the second term on the right side of \eqref{anz2}.
\end{prop}

\begin{proof} Since $c2_k=-b2_k(m\to m+1)$, we have 
by Proposition~\ref{altb2}
$$
\begin{aligned}
\sum_{k=1}^{m+1} c2_k &= -  q^{-m-1}\sum_{i=1}^{m+1} 
\frac{(-1)^{i-1}q^{-i(i+1)}}{(1/q^2;1/q^2)_{m+1-i}} q^{2i}\\
&= q^{-m-1}\sum_{i=0}^{m} 
\frac{(-1)^{i-1}q^{-i(i+1)}}{(1/q^2;1/q^2)_{m-i}}.
\end{aligned}
$$
\end{proof}

\begin{prop} \label{part02} We have
$$
\sum_{k=1}^{m+1} c1_k=  \frac{1}{q^m (q^{-2};q^{-2})_m}.
$$
\end{prop}
\begin{proof} We have
$$
\begin{aligned}
\sum_{k=1}^{m+1} c1_k &= q^{-m}  \sum_{k=0}^m 
\frac{(q^{-2m};q^2)_k}{(q^{-2};q^{-2})_k^2} q^{-2k^2}\\
&=q^{-m} \sum_{k=0}^m \frac{(q^{2m};q^{-2})_k}{(q^{-2};q^{-2})_k^2}(-1)^k q^{-k(k-1)-2mk-2k}\\
&= q^{-m}\lim_{B\to\infty} \ _2\phi_1(q^{2m},B;q^{-2};q^{-2},q^{-2m-2}/B)\\
&= q^{-m} \lim_{B\to\infty} \frac{(q^{-2}/B;q^{-2})_m}{(q^{-2};q^{-2})_m}\\
&= \frac{1}{q^m(q^{-2};q^{-2})_m}.
\end{aligned}
$$
where in the penultimate step, we used Proposition~\ref{qchu}.
\end{proof}

Combining Propositions \ref{part01} and \ref{part02} completes the proof of \eqref{eq5}, and so of Conjecture
\eqref{anz2}.

\section{Anzanello's third conjecture} \label{Oproof2}

The proof of Conjecture \eqref{anz3} is not so difficult given our earlier results. We use the probability distribution on
$P_O$ defined in Section \ref{Oproof1}. It was proved in \cite{F1} that

\[ \sum_{\lambda_1'=2k} P_O(\lambda) = \frac{\prod_{i \geq 1} (1-u^2/q^{2i-1})}{(1+u)} \frac{u^{2k}}{q^{2k^2-k} (1/q^2;1/q^2)_k
(u^2/q;1/q^2)_k}.\]

It follows that 

\begin{eqnarray*}
& & \sum_{|\lambda|=2m, \lambda_1'=2k \atop i \ even \implies m_i(\lambda) \ even} \frac{1}{q^{\frac{1}{2} \sum_i (\lambda_i')^2 - \frac{o(\lambda)}{2}}  
\prod_i (1/q^2;1/q^2)_{\lfloor \frac{m_i(\lambda)}{2} \rfloor}} \\
& = & \frac{1}{q^{2k^2-k} (1/q^2;1/q^2)_{k}} \times \rm{Coeff. \ } u^{2m} \rm{\ in \ }  \frac{u^{2k}}{(u^2/q;1/q^2)_k} \\
& = & \frac{1}{q^{2k^2-k} (1/q^2;1/q^2)_{k}} \times \rm{Coeff. \ } u^{m-k} \rm{\ in \ }  \frac{1}{(u/q;1/q^2)_k}.
\end{eqnarray*} 

Note that if $|\lambda|$ is even and $P_O(\lambda) \neq 0$, then since the even parts of $\lambda$ occur with even multiplicity, $\lambda_1'$ must
be even. Hence the left side of Conjecture \eqref{anz3} is equal to $\sum_{k=1}^m d_k$, where
\[ d_k = \frac{1}{q^{2k^2-k} (1/q^2;1/q^2)_{k-1}} \times \rm{Coeff. \ } u^{m-k} \rm{\ in \ } \frac{1}{(u/q;1/q^2)_k}.\] Since
$d_k = b2_k$, it follows from Proposition \ref{altb2} that that \[ \sum_{k=1}^m d_k = \frac{1}{q^m} \sum_{i=1}^m \frac{(-1)^{i-1}}{q^{i(i-1)} (1/q^2;1/q^2)_{m-i}},\] completing the proof of Conjecture \eqref{anz3}.

\section{Acknowledgements} Fulman was supported by Simons Foundation grant 917224.


\begin{thebibliography}{A}

\bibitem{An} Anzanello, J., On the proportion of derangements in affine classical groups, arXiv:2508.07093 (2025).

\bibitem{BG} Burness, T. and Giudici, M., Classical groups, derangements, and primes, Cambridge University
Press, 2016.

\bibitem{CL} Cohen, H. and Lenstra, H.W., Heuristics on class groups of number fields, in: Number theory, Noordwijkerhout, 1983,
in {\it Lecture Notes in Math, Volume 168}, Springer, Berlin, 1984, 33-62.

\bibitem{DFG} Diaconis, P., Fulman, J. and Guralnick, R., On fixed points of permutations, {\it J. Algebraic
Combin.} {\bf 28} (2008), 189-218.

\bibitem{EV} Ellenberg, J. and Venkatesh, A., Statistics of number fields and function fields, Proceedings of the 2010 International
Congress of Mathematicians, (2011), 383-402.

\bibitem{F1} Fulman, J., A probabilistic approach to conjugacy classes in the finite symplectic and orthogonal groups,
{\it J. Algebra} {\bf 234} (2000), 207-224.

\bibitem{F2} Fulman, J., Cohen-Lenstra heuristics and random matrix theory over finite fields, {\it J. Group Theory}
{\bf 17} (2014), 619-648.

\bibitem{GR} Gasper, G. and Rahman, M., Basic hypergeometric series, Encyclopedia of Mathematics and its Applications,
35, Cambridge University Press, 1990.

\bibitem{Se} Serre, J.-P., On a theorem of Jordan, {\it Bull. Amer. Math. Soc. (N.S.)} {\bf 40} (2003), 429-440.

\bibitem{Wood} Wood, M., Probability theory for random groups arising in number theory, Proceedings of the 2022 International
Congress of Mathematicians, (2023), 4476-4508.
\end{thebibliography}
\end{document}